\documentclass[12pt]{amsart}
\usepackage[margin=1in]{geometry}
\usepackage{latexsym}
\usepackage{float}
\usepackage{amssymb}
\usepackage[all]{xy}
\usepackage{epsfig}
\usepackage{color}
\usepackage{graphics}
\usepackage{hyperref}

\newtheorem{Thm}[equation]{Theorem}
\newtheorem{Lem}[equation]{Lemma}
\newtheorem{Cor}[equation]{Corollary}

\begin{document}

\title{ON THE CLASSIFICATION OF A LARGE SET OF RATIONAL $3$-TANGLE DIAGRAMS}
\author{ Bo-hyun Kwon}
\date{Friday, February 19, 2015}
\maketitle

\begin{abstract}

We note that a rational $3$-tangle diagram is obtained from a combination of four generators. There is an algorithm to distinguish two rational $3$-tangle diagrams up to isotopy.(\cite{3}) However, there is no perfect classification about  rational $3$-tangle diagrams such as the classification of rational $2$-tangle diagrams corresponding to rational numbers.
In this paper, we classify a large set of rational $3$-tangles which are generated by only three generators.
\end{abstract}

\section{Introduction}

A \textit{3-tangle} is the disjoint union of 3 properly embedded arcs in the unit 3-ball. The embed-
ding must send the endpoints of the arcs to 6 fixed points on the ball's boundary.
Without loss of generality, consider the fixed points on the 3-ball boundary to lie on a
great circle. The tangle can be arranged to be in general position with respect to the projection onto the 
at disk in the $xy$-plane bounded by the great circle. The projection then gives
us a tangle diagram where we make note of over and undercrossings as with knot diagrams.
Then we say that a 3-tangle $\alpha_1\cup\alpha_2\cup\alpha_3$
 in $B^3$ is \textit{rational} if there exists a homeomorphism
of pairs
$\bar{h}:(B^3,\alpha_1,\alpha_2,\alpha_3)\rightarrow(D^2\times I,\{x_1,x_2,x_3\}\times I)$.
Also, we define that two rational 3-tangles, $\mathbb{T},\mathbb{T}'$, in $B^3$ are \textit{isotopic},  denoted by $\mathbb{T}\approx \mathbb{T}'$, if there is an orientation-preserving self-homeomorphism $h:(B^3,\mathbb{T})\rightarrow (B^3,\mathbb{T}')$
that is the identity map on the boundary. Also we say that two rational $3$-tangle diagrams $T$ and $T'$ are \textit{isotopic}, denoted by $T\sim T'$, if there exist two rational $3$-tangles $\mathbb{T}$ and $\mathbb{T}'$ so that $\mathbb{T}\approx \mathbb{T}'$ and $T$ and $T'$ are projections of $\mathbb{T}$ and $\mathbb{T}'$ respectively.\\

In 1970,  Conway~\cite{7} introduced tangles and he proved that two rational 2-tangles
are isotopic if and only if they have the same rational number.  However, there is no similar invariant known which classifies rational $3$-tangles.   Currently, the author~\cite{3} found an algorithm to check whether or not two rational 3-tangle diagrams are isotopic  by using a modified version of Dehn's method for classifying simple closed curves on surfaces. 
 \\

\begin{figure}[htb]
\includegraphics[scale=.7]{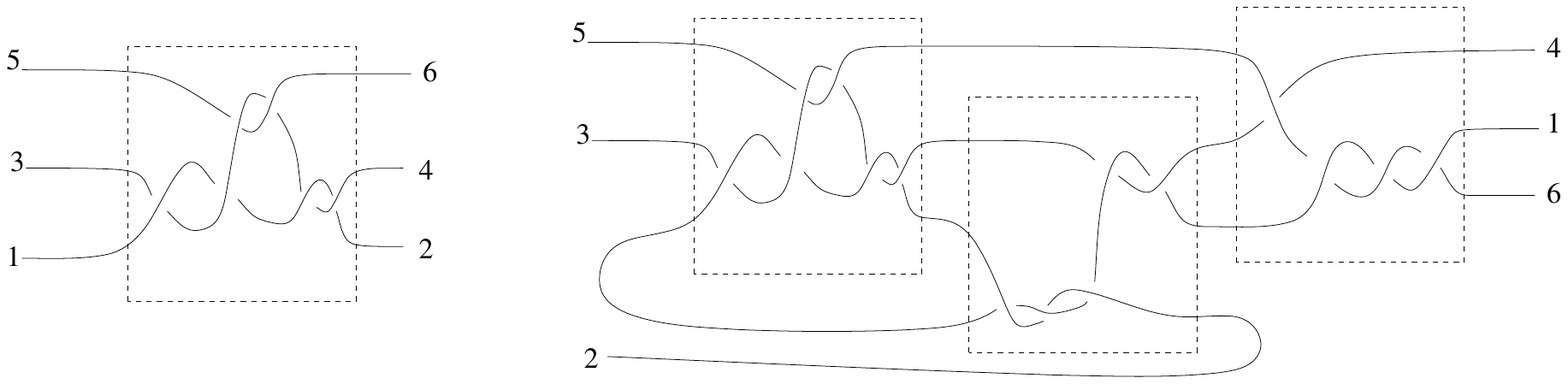}
\vskip -480pt
\caption{Two examples in $\mathcal{T}_{\{\hat{\sigma_1},\hat{\sigma_2},\hat{\sigma_3}\}}$}
\label{p1}
\end{figure}
We note that the rational 3-tangles are obtained  by four generators extended from four half Dehn twists on $\Sigma_{0,6}$. (Refer to~\cite{3}.) Especially,  if we use two proper generators of them then we  obtain a family of all the rational $3$-tangle diagrams presented by the braids group with three strings, $\mathbb{B}_3$, such as the first diagram in Figure~\ref{p1}. Recently, Cabrera-Ibarra (\cite{5}, \cite{6}) found a pair of invariants which is defined for all rational 3-tangles. Each invariant is a $3\times 3$ matrix with complex number entries.
  The pair of invariants classifies the elements of six special sets of rational 3-tangles each of which contains the braid 3-tangles.
 In this paper, we are investigating  a large set of rational 3-tangles,  $\mathcal{T}_{\{\hat{\sigma_1},\hat{\sigma_2},\hat{\sigma_3}\}}$,  which are generated by the three generators $\hat{\sigma_1},\hat{\sigma_2},\hat{\sigma_3}$, where $\hat{\sigma_1},\hat{\sigma_2},\hat{\sigma_3}$ are extensions of the three half Dehn twists $\sigma_1,\sigma_2,\sigma_3$ on $\Sigma_{0,6}$. (See Figure~\ref{p6}.) We note that our family of rational $3$-tangle diagrams contains at least five of the special sets by Cabrera-Ibarra. 
     We also note that our set contains lots of rational 3-tangle diagrams which are not in the six special sets.\\

The following is the main corollary of this paper that classifies
  $\mathcal{T}_{\{\hat{\sigma_1},\hat{\sigma_2},\hat{\sigma_3}\}}$. We will prove this corollary in Section 3.\\

\textbf{Corollary~\ref{C}} \hskip 10pt Suppose that two rational 3-tangle diagrams $T_F, T_G\in \mathcal{T}_{\{\hat{\sigma_1},\hat{\sigma_2},\hat{\sigma_3}\}}$.  $T_F\sim T_G$ if and only if  $[F(\partial E_1)]= [G(\partial E_1)]$ and $(F(\epsilon_2)\cup F(\epsilon_3),B^3)\approx (G(\epsilon_2)\cup G(\epsilon_3),B^3)$.\\

\section{Dehn parameterization of $\mathcal{C}$}\label{S2}

\begin{figure}[htb]
\includegraphics[scale=.35]{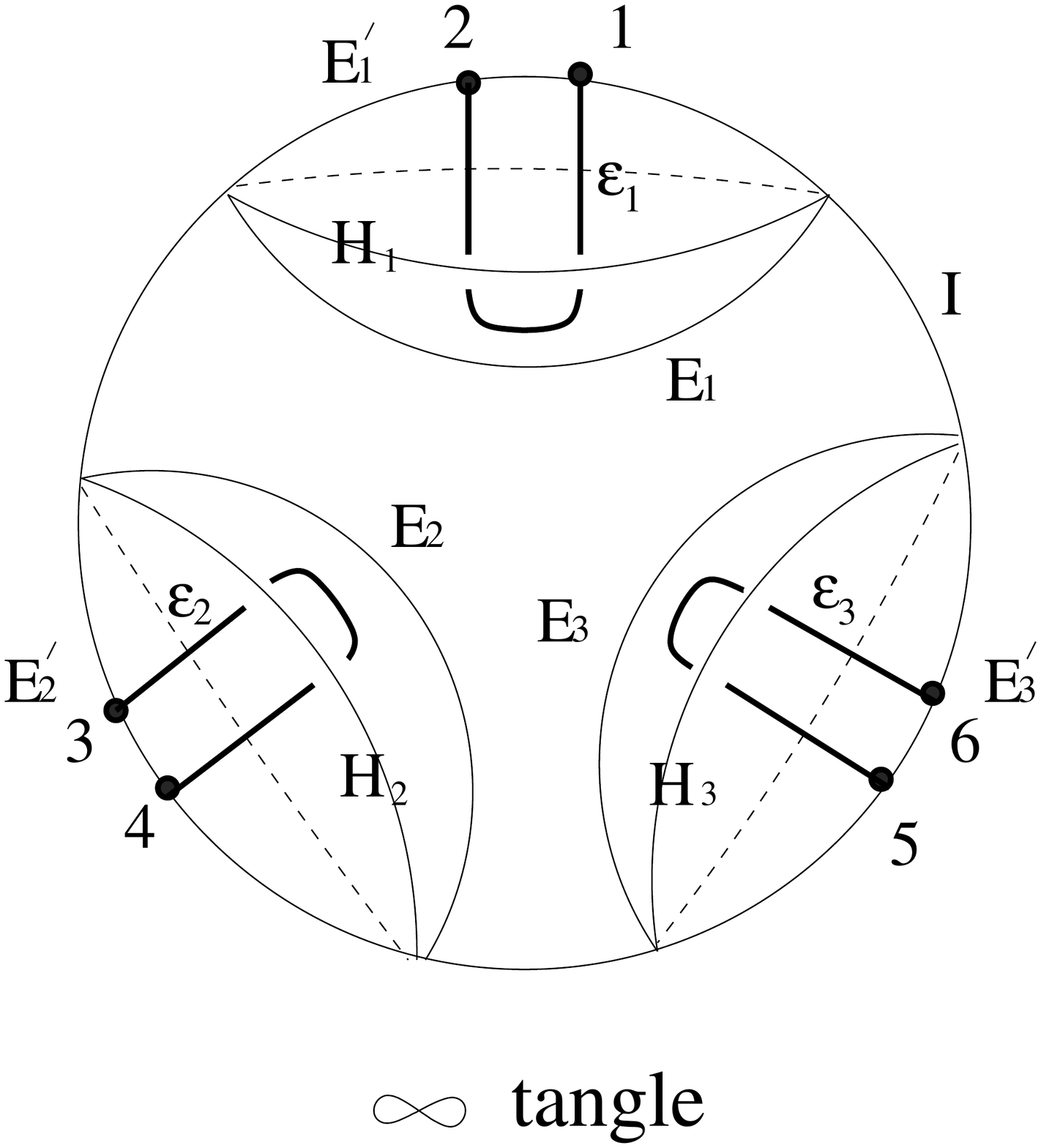}
\vskip -80pt
\caption{}
\label{p2}
\end{figure}

Let $\Sigma_{0,6}$ be the six punctured sphere and $(B^3,\epsilon)=(B^3,\epsilon_1\cup\epsilon_2\cup\epsilon_3)$  be the trivial tangle (or $\infty$ tangle) as in Figure~\ref{p2}. Also, let $E_i$ the the three disjoint essential disks in $B^3-\epsilon$ as in Figure~\ref{p2}. Then we have the three $2$-punctured disks $E_i'$ in $\Sigma_{0,6}$ so that $\partial E_i'=\partial E_i$. Then we note that $E_i\cup E_i'$ bounds the ball $H_i$ which only contains $\epsilon_i$. 
Let $E=E_1\cup E_2\cup E_3$ and $E'=E_1'\cup E_2'\cup E_3'$.
Also, let $\partial E=\partial E'=\partial E_1\cup \partial E_2\cup\partial E_3$. Then we can get the pair of pants $I=\overline{(\Sigma_{0,6}-E')}$.

\begin{figure}[htb]

\includegraphics[scale=.5]{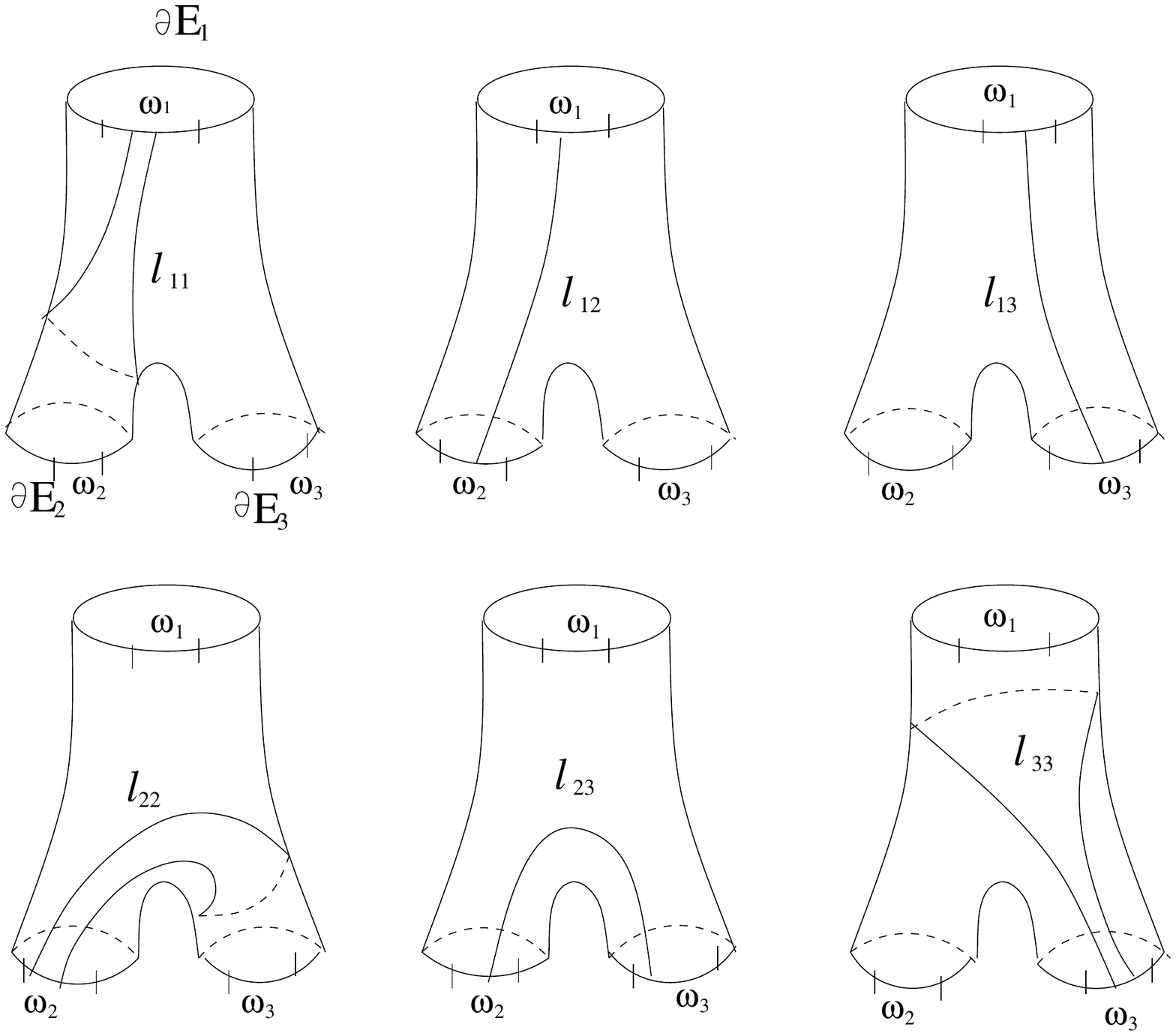}
\vskip -175pt
\caption{Standard arcs $l_{ij}$}
\label{E1}
 \end{figure}

Let $\gamma$ be a simple closed curve in $\Sigma_{0,6}$. 
\\

Figure~\ref{E1}  shows $\textit{standard arcs}$ $l_{ij}$ in the pair of pants $I$. We notice that we can isotope $\gamma$ into $\delta$ in $\Sigma_{0,6}$ so that each component of $\delta\cap I$ is isotopic to one of the standard arcs and $\delta\cap \partial E_i\subset \omega_i$. 
Then we say that subarc $\alpha$ of $\delta$ $\textit{is carried by}$ $l_{ij}$ if some component of $\alpha\cap I$ is isotopic to $l_{ij}$. The closed arc $\omega_i\subset \partial E_i$ is called a $\textit{window}$. \\

 Let $I_i=|\delta\cap \omega_i|$. 
Then $\delta$ can have many parallel arcs which are the same type in $I$. Let $x_{ij}$ be the number of parallel arcs of the type $l_{ij}$ which is called the $weight$ of $l_{ij}.$ \\

\begin{figure}[htb]
\begin{center}
\includegraphics[scale=.8]{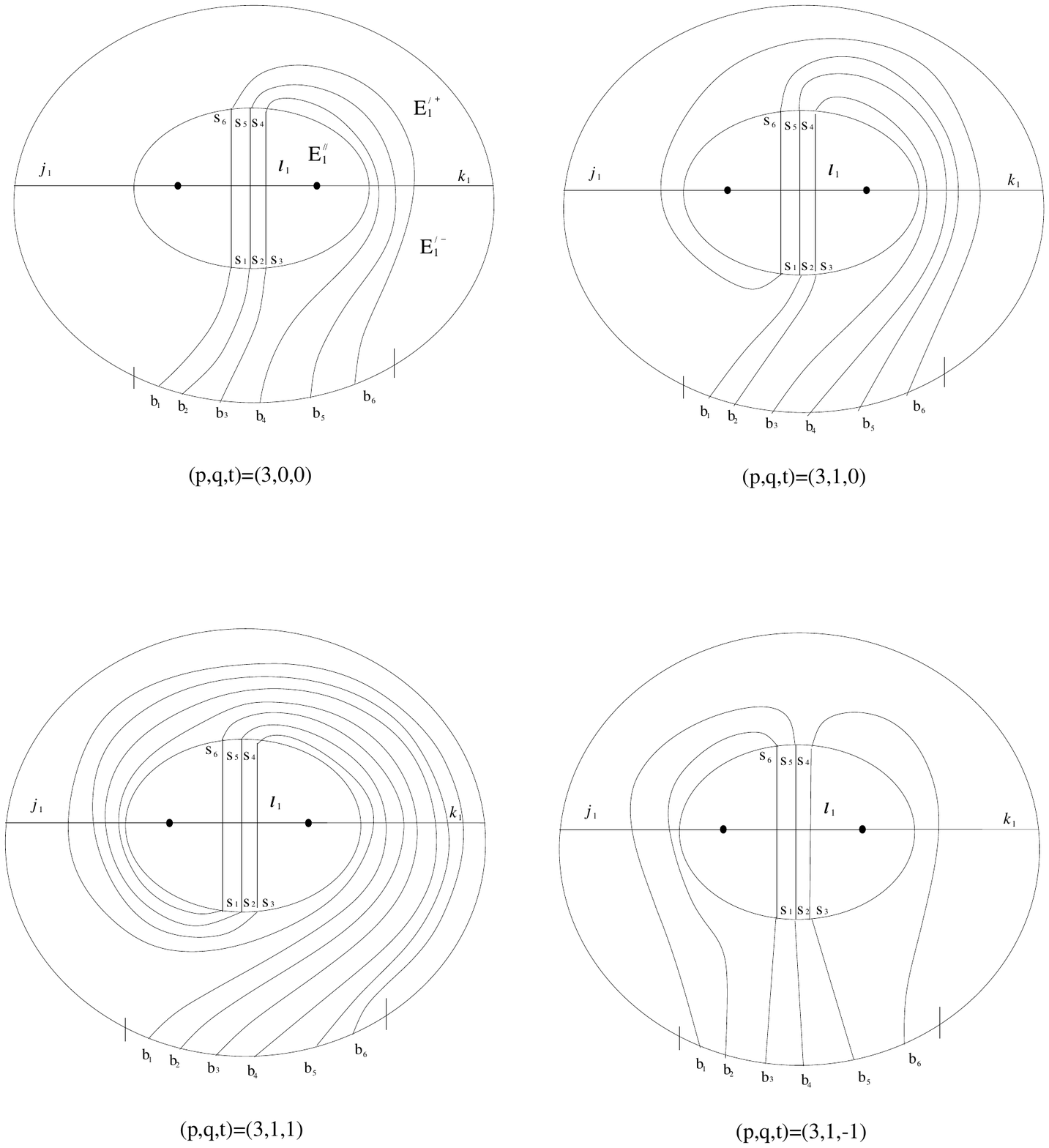}
\end{center}
\vskip -200pt
\caption{}
\label{E2}
 \end{figure}
 
We define a bigon as follows. For two simple subarcs $\lambda$ and $\mu$ of a collection $K$ of simple closed curves in a surface $\Sigma$, $(\Delta,\lambda,\mu)$ is a \emph{bigon} in a surface $\Sigma$ if $\lambda\cup\mu$ bounds a disk $\Delta$ in $\Sigma$ and $\partial \lambda=\partial \mu=\lambda\cap \mu$. Then we say that a bigon $\Delta$ is \emph{innermost} if $\Delta^\circ\cap K=\emptyset$.\\

Now, consider $E_1'$.   Let $j_1$ and $k_1$ be the simple arcs as in Figure~\ref{E2}.
We assume that $\partial E_1'\cup\delta \cup j_1\cup k_1\cup l_1$ has no bigon in $\Sigma_{0,6}$. We note that $j_1\cup k_1\cup l_1$ separates $E_1'$ into two semi-disks ${E_1'}^+$ and ${E_1'}^-$ as in Figure~\ref{E2}.
Let ${u}_1^+$ be the number of subarcs of $\delta$ from $l_1$ to $j_1$ in  ${E_1'}^+$. Also, let ${v}_1^+$ be the number of subarcs of $\delta$ from $l_1$ to $k_1$ in ${E_1'}^+$ and let ${w}_1^+$ be the number of subarcs of $\delta$ from $j_1$ to $k_1$ in ${E_1'}^+$.  
Let $m_1=|\delta\cap j_1|$ and $n_1=|\delta\cap k_1|$ in $E_1'$.  For example, in the third diagram of Figure~\ref{E2}, we have ${u}_1^+=0$, ${v}_1^+=3$, ${w}_1^+=4$, $m_1=4$ and $n_1=7$.\\

 We notice that each component of $\delta\cap E_1'$ meets $l_1$ exactly once. Also, we know that each such component is essential in 
$E_1'$. Let $E_1''$ be the nested $2$-punctued disk in $E_1'$ as in Figure~\ref{E2}. We can isotope $\delta$ in $E_i'$ so that all the components of $\delta\cap E_i''$ are pairwise parallel. Then, let $I'=\Sigma_{0,6}-\{E_1''\cup E_2''\cup E_3''\}$. \\

The components of $\delta\cap E_1'$ are determined by three parameters $p_1,q_1,t_1$ as in Figure~\ref{E2},
where $p_1=\min\{|\delta'\cap l_1||\delta'\sim \delta$ in $\Sigma_{0,6}\}$, $q_1\in \mathbb{Z},~0\leq q_1<p_1$. 
In order to define $q_1$ and $t_1$, consider $m_1$ and $n_1$.
 Then we know that $u_1^++v_1^+=p_1$. 
So, $m_1-n_1=(u_1^++w_1^+)-(v_1^++w_1^+)=u_1^+-v_1^+$. Therefore, $-p_1=-u_1^+-v_1^+\leq u_1^+-v_1^+=m_1-n_1=u_1^+-v_1^+\leq u_1^++v_1^+=p_1$.
So, we know $-p_1\leq m_1-n_1\leq p_1$. Now, we define $q_1$ and $t_1$ as follows. If $n_1-m_1=p_1$ then $q_1\equiv m_1$ (mod $p_1$) and $0\leq q_1<p_1$, and $t_1={m_1-q_1\over p_1}$ and if $-p_1\leq n_1-m_1<p_1$ then $q_1\equiv -m_1$ (mod $p_1$) and $0\leq q_1<p_1$, and $t_1={-m_1-q_1\over p_1}$.
Then $t_1$ is called the $twisting$ $number$ in $E_1'$. 
Also, let $(p_1,q_1,t_1)$ be the three parameters to determine the arcs in $E_1'$. Similarly, we have the three parameters $(p_i,q_i,t_i)$ for $E_i'$ ($i=2,3$). 
Then $\gamma$ is determined by a sequence of nine parameters $(p_1,q_1,t_1,p_2,q_2,t_2,p_3,q_3,t_3)$ by Lemma~\ref{T51}.

\begin{figure}[htb]
\begin{center}
\includegraphics[scale=.7]{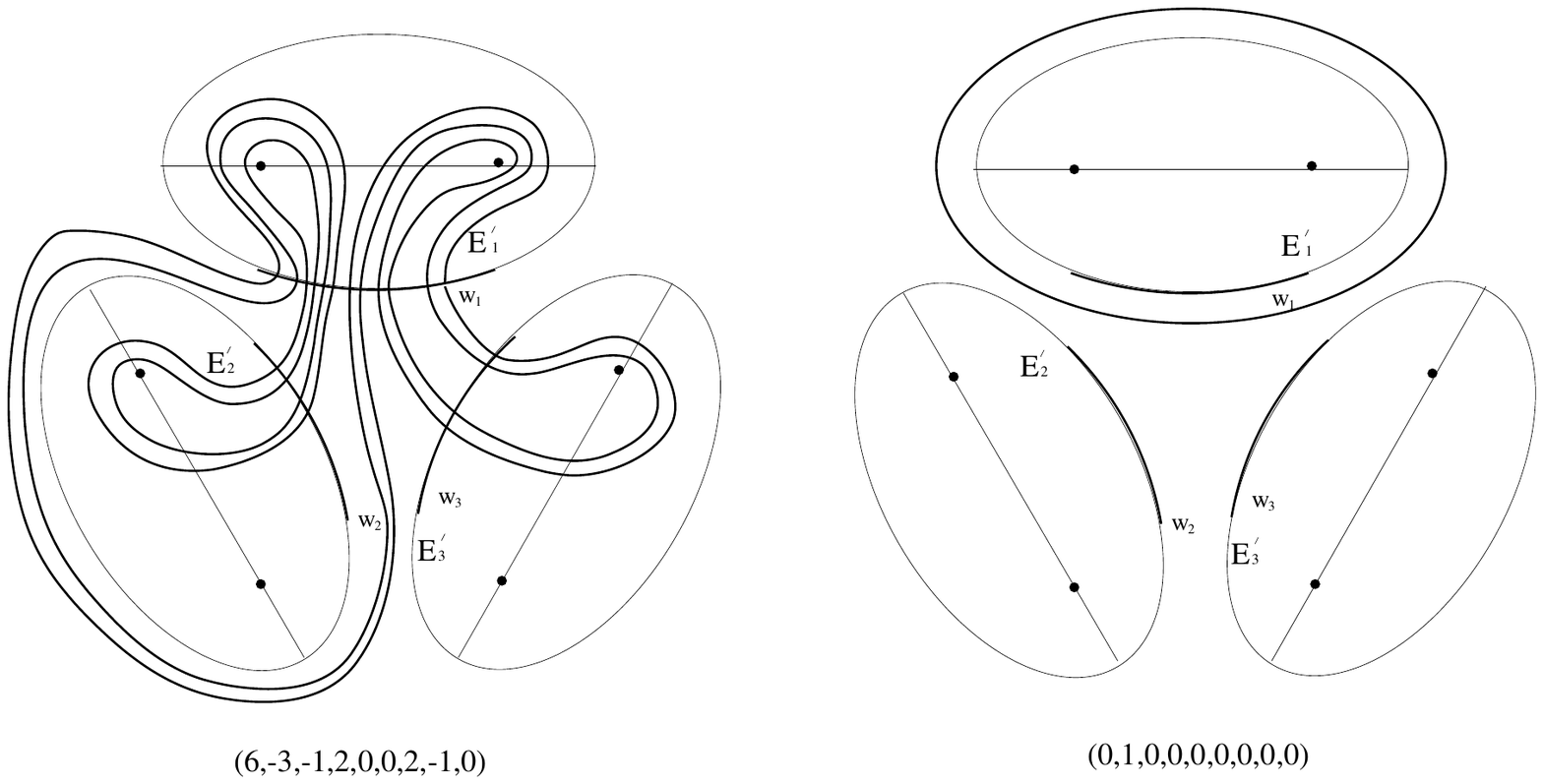}
\end{center}
\vskip -400pt
\caption{Essential curves obtained by a sequece of nine parameters}
\label{E3}
\end{figure}

\begin{Lem}[\cite{3}]\label{T51}

From $I_i$ ($i=1,2,3$), we have $x_{ij}=I_j,~ x_{ik}=I_k$ and $x_{ii}={I_i-I_j-I_k\over 2}$.
\end{Lem}





Let $\mathcal{C}$ be the set of isotopy classes of simple closed curves in $\Sigma_{0,6}$.
For a given simple closed curve $\delta$ in a hexagon diagram, we define $p_i$, $q_i$ and $t_i$ in $E_i'$ as above. Then 
let $q_i'=p_it_i+q_i$ for $i=1,2,3$. 
\begin{Thm} [Special case of Dehn's Theorem ]\label{T52}

There is an one-to-one map $\phi:\mathcal{C}\rightarrow\mathbb{Z}^6$ so that $\phi(\delta)=(p_1,p_2,p_3,q_1',q_2',q_3')$. i.e., it classifies isotopy classes of simple closed curves.
 \end{Thm}

When $p_1=p_2=p_3=0$ then $t_i'=1$ if the simple closed curve is isotopic to $\partial E_i'$ and $t_j'=0$ if $j\neq i$.
Refer~\cite{4} to see the general Dehn's theorem.\\

We will use a sequence of nine parameters instead of six parameters for convenience in the rest of this paper.\\

\begin{Lem}\label{T1}
If $\gamma$ bounds an essential disk in $B^3-\epsilon$ and $|\gamma\cap \partial E|\neq \emptyset$, then $x_{ii}\geq 2$ for some $i\in\{1,2,3\}$.
\end{Lem}

\begin{proof}
Let $D$ be an essential disk in $B^3-\epsilon$ so that $
\partial D =\gamma$ . Then we know that $E_i$ will
separates $D$ into some sub-disks $D_i$ since $|\gamma\cap\partial E|\neq \emptyset$. We notice that at least two of $D_i$
need to be bounded by a bigon. One side of bigons needs to be a component of 
 $\gamma\cap I$ and
another side of bigons is a properly embedded arc in $E_i$ for some $i$ since $H_i$ cannot have a
bigon. If $H_i$ has a bigon then the disk $\Delta$ bounded by the bigon should meet $\epsilon_i$. This makes
a contradiction. So, we need $x_{11} + x_{22} + x_{33} \geq 2$. However, $l_{ii}$ and $l_{jj}$ cannot coexist for
$i \neq j$. This implies that $x_{ii} \geq 2$ for some $i \in \{1,2,3\}.$
\end{proof}

Consider the three half Dehn twists $\sigma_1,\sigma_2, \sigma_3$ and $\sigma_4$ supported on the $2$-punctured disks $A,B,E_1'$ and $E_3'$ counter clockwise as in Figure~\ref{p6}. Then let $\hat{\sigma_i}$ be an isotopy extension to $B^3$ of $\sigma_i$ for $i=1,2,3$. Then we note that $\hat{\sigma_1},\hat{\sigma_2},\hat{\sigma_3},\hat{\sigma_4}$ generates all the rational $3$-tangles. (Refer to~\cite{3}.)\\

Let $\mathcal{T}_{\{\hat{\sigma_1}\,\hat{\sigma_2},\hat{\sigma_3}\}}$ be the set of rational 3-tangles which are generated by $\hat{\sigma_1}\,\hat{\sigma_2},\hat{\sigma_3}$. 
Actually, the main goal of this paper is the classification of $\mathcal{T}_{\{\hat{\sigma_1}\,\hat{\sigma_2},\hat{\sigma_3}\}}$.

\begin{figure}[htb]
\includegraphics[scale=.6]{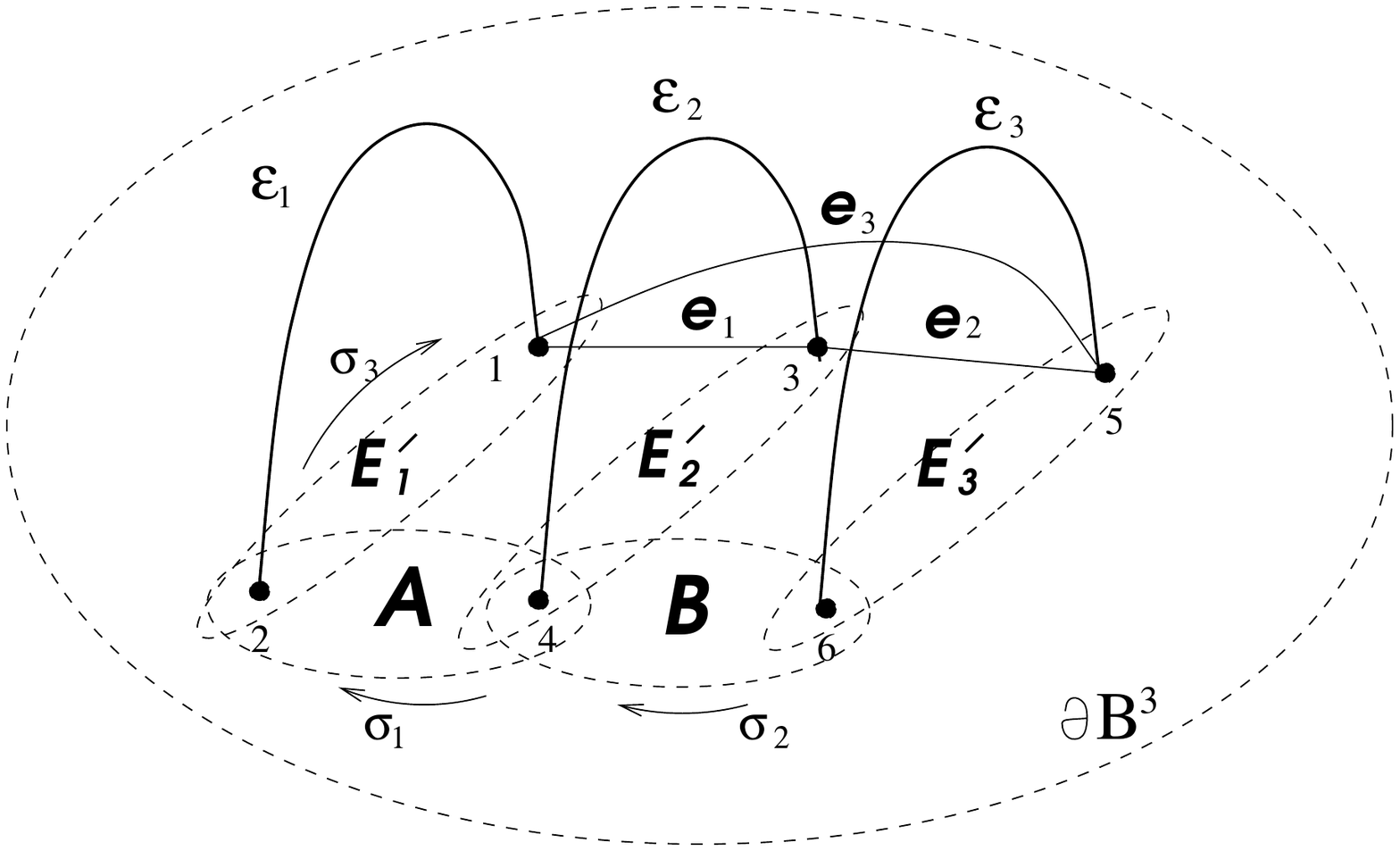}
\vskip -300pt
\caption{}
\label{p6}
\end{figure}

Let $f$ and $g$ be orientation preserving homeomorphisms from $\Sigma_{0,6}$ to $\Sigma_{0,6}$. Then we have isotopy extensions $F$ and $G$  of $f$ and $g$ respectively to $B^3$. Let $\mathbb{T}_F=(F(\epsilon),B^3)$ and $\mathbb{T}_G=(G(\epsilon),B^3)$. Then we have the following theorem.

\begin{Thm}[\cite{3}]\label{T2}
For two rational 3-tangles $\mathbb{T}_F$ and $\mathbb{T}_G$,
$\mathbb{T}_F\approx \mathbb{T}_G$ if and only if $G^{-1}F(\partial E)$ bounds essential disks in $B^3-\epsilon$.
\end{Thm}

Actually, if $G^{-1}F(\partial E_i)$ and $G^{-1}F(\partial E_i)$  bound essential disks in $B^3-\epsilon$ for $i\neq j$ then $G^{-1}F(\partial E_k)$ also bounds an essential disk in $B^3-\epsilon$ for $k\neq i,j$. So, it is enough to check whether two of them bound essential disks or not. The following lemma supports this argument.

\begin{Lem}[\cite{3}]\label{T3}
Suppose that two essential simple closed curves $\alpha,
\beta$ ($\nsim \alpha $) bound disjoint disks
in $B^3-\epsilon$. If 
$\gamma$ is an essential simple closed curve which encloses two punctures, disjoint
with $\alpha$ and $\beta$ and non-parallel to $\alpha$ and $\beta$, then $\gamma$
 bounds an essential disk in $B^3-\epsilon$.
\end{Lem}

Let $\gamma_1$ and $\gamma_2$ be disjoint two simple closed curves in $\Sigma_{0,6}$. Then we consider a rectangle $R$ (or band) in $\Sigma_{0,6}$ so that the interior of $R$ is disjoint with $\gamma_1\cup\gamma_2$ and $R_i=\gamma_i\cap R$ is a side of $R$ for $i=1,2$. Let $R_3$ and $R_4$ be the other two sides of $R$.\\

Now, we define the $\emph{band sum}$  of $\gamma_1$ and $\gamma_2$ by $R$, denoted by $\gamma_1+_R\gamma_2$, which is $(\gamma_1\cup\gamma_2)\cup (R_3\cup R_4)-(R_1^\circ\cup R_2^\circ)$. Then we have the following lemma.

\begin{Lem}\label{T4}
Let $\gamma_1$ and $\gamma_2$ be two disjoint, non-parallel simple closed curves in $\Sigma_{0,6}$ which enclose only two punctures of $\{1,2,3,4,5,6\}$. Then $\gamma_1+_R\gamma_2$ is  disjoint with $\gamma_1$ and $\gamma_2$ up to isotopy and it is a non-parallel simple closed curve to $\gamma_1$ and $\gamma_2$ which encloses only two punctures of $\{1,2,3,4,5,6\}$.
\end{Lem}

\begin{proof}
Take $\gamma_i'$ which is parallel to $\gamma_i$ and the $2$-punctured disk bounded by $\gamma_i'$ contains $\gamma_i$. We also need to assume that $\gamma_1'$ and $\gamma_2'$ are disjoint. Then we can take $R'$ so that $\gamma_1'+_{R'}\gamma_2'$ is isotopic to $\gamma_1+_R\gamma_2$ and $\gamma_1'+_{R'}\gamma_2'$ is disjoint with both $\gamma_1$ and $\gamma_2$.\\

Let $\{a,b\}$ and $\{c,d\}$ be the two punctures which are enclosed in $\gamma_1$ and $\gamma_2$ respectively.
Since $\gamma_1$ and $\gamma_2$ are two disjoint, non-parallel simple closed curves in $\Sigma_{0,6}$, they have $\{a,b\}\cap\{c,d\}=\emptyset$. Then we note that $\gamma_1+_R\gamma_2$ enclose the two punctures $\{1,2,3,4,5,6\}-\{a,b,c,d\}$. This implies that $\gamma_1+_R\gamma_2$ are not parallel to both $\gamma_1$ and $\gamma_2$.
\end{proof}

Now, we consider the ``standard parameterization" of $\mathcal{C}$ for easier argument in the next section.
\section{Standard parameterization of $\mathcal{C}$ and the proof of  main theorem}

First, let $\gamma$ be a simple closed curve which bounds an essential disk in $B^3-\epsilon$. Then by Lemma~\ref{T1}, $x_{ii}\geq 2$ for some $i\in\{1,2,3\}$. Without loss of generality, assume that $x_{11}\geq 2$. So, $x_{22}=x_{33}=0$.\\

Now, we modify $\gamma$ into $\gamma_0$ so that $\gamma_0$ bounds an essential disk in $B^3-\epsilon$ if and only if $\gamma$ does and every components of $\gamma_0\cap P'$ are isotopic to one of arcs in the given diagrams as in Figure~\ref{p7} which are called $\emph{standard diagrams}$. Also, we say that $\gamma_0$ is in $\emph{standard position}$ if $\gamma_0$ has a standard diagram.

 \begin{figure}[htb]
 \includegraphics[scale=.5]{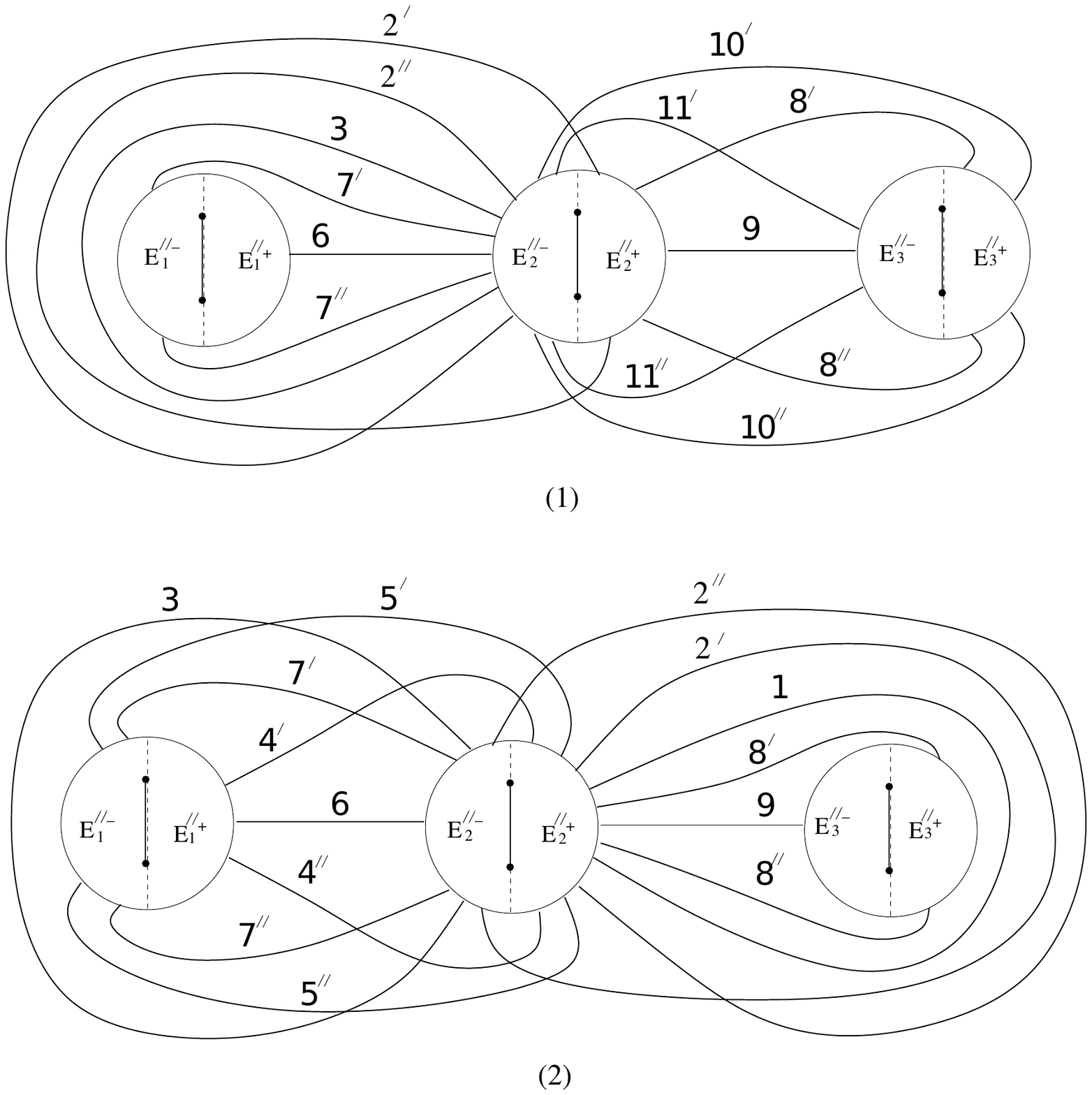}
 \vskip -170pt
 \caption{Standard diagrams}
 \label{p7}
 \end{figure}

The following lemma shows us how to modify $\gamma$ into $\gamma_0$ so that $\gamma_0$ has a standard diagram. 

\begin{Lem}[\cite{3}] Suppose that $\gamma'$ is a simple closed curve which is parameterized by $(p_1,q_1,0,p_2,q_2,$ $0,p_3,q_3,0)$.
If $x_{11}>0$,
then we can construct a  simple closed curve $\gamma_0$ which is parameterized by $(p_1,q_1,t_1,p_2,q_2,t_2,p_3,q_3,t_3)$ for  $t_i\in \mathbb{Z}$ as in the table below, and it bounds an essential disk in $B^3-\epsilon$ if  $\gamma'$ does. Moreover, each component of $\gamma_0\cap I'$ is carried by one of the given arc types in one of the standard diagrams.

\begin{enumerate}
\item $q_1+p_1< x_{11}+x_{13}:$ $(t_1,t_2,t_3)=(0,-1,0)$ if $p_2\neq 0$, $(t_1,t_2,t_3)=(0,0,0)$ if $p_2=0$.\\

\item $ q_1+p_1 \geq x_{11}+x_{13}:$ $(t_1,t_2,t_3)=(-1,-1,0)$ if $p_2\neq 0$, $(t_1,t_2,t_3)=(-1,0,0)$ if $p_2=0$.

\end{enumerate}

Moreover, if we have the following condition then $\gamma'$ does not bound an essential disk in $B^3-\epsilon$.
\begin{enumerate}
\item [(3)] $x_{11}+x_{13}\leq q_1+p_1<x_{11}+x_{12}+x_{13}$ and $x_{13}\geq q_1$.

\end{enumerate}
\end{Lem}

 Then $\gamma_0$ should satisfy the following condition to bound an essential disk in $B^3-\epsilon$.
 
 \begin{Lem}[\cite{3}]\label{T32}
Suppose that $\gamma_0$ is a simple closed curve which bounds an essential disk $A$ in $B^3-\epsilon$ and it is in standard position with $x_{11}>0$.
 Then  $m_1+m_3>0$.
\end{Lem} 

Let $m_i$ be the number of the arcs which are isotopic to arc type $i$. We note that there are two non-isotopic arc types $i'$ and $i''$ for some arc type $i$, where $i'$ is the upper arc type of $i$ and $i''$ is the lower arc type of $i$. Then we define $m_{i'}$ and $m_{i''}$ which are the weight of each of the arc types. So, $m_i=m_{i'}+m_{i''}$.\\

Now, we have the following three lemmas which support the main theorem.

\begin{Lem}[\cite{3}]\label{T33}
 $\gamma_0$ bounds an essential disk in $B^3-\epsilon$ if and only if $(\delta_1\delta_2^{-1})^{\pm 1}(\gamma_0)$  does, where $\delta_1$ and $\delta_2$ be half Dehn twists counter clockwise supported on two punctured disks $C_1$ and $C_2$ respectively. (Refer to Figure~\ref{p8}.)
\end{Lem}

\begin{figure}[htb]
\includegraphics[scale=.5]{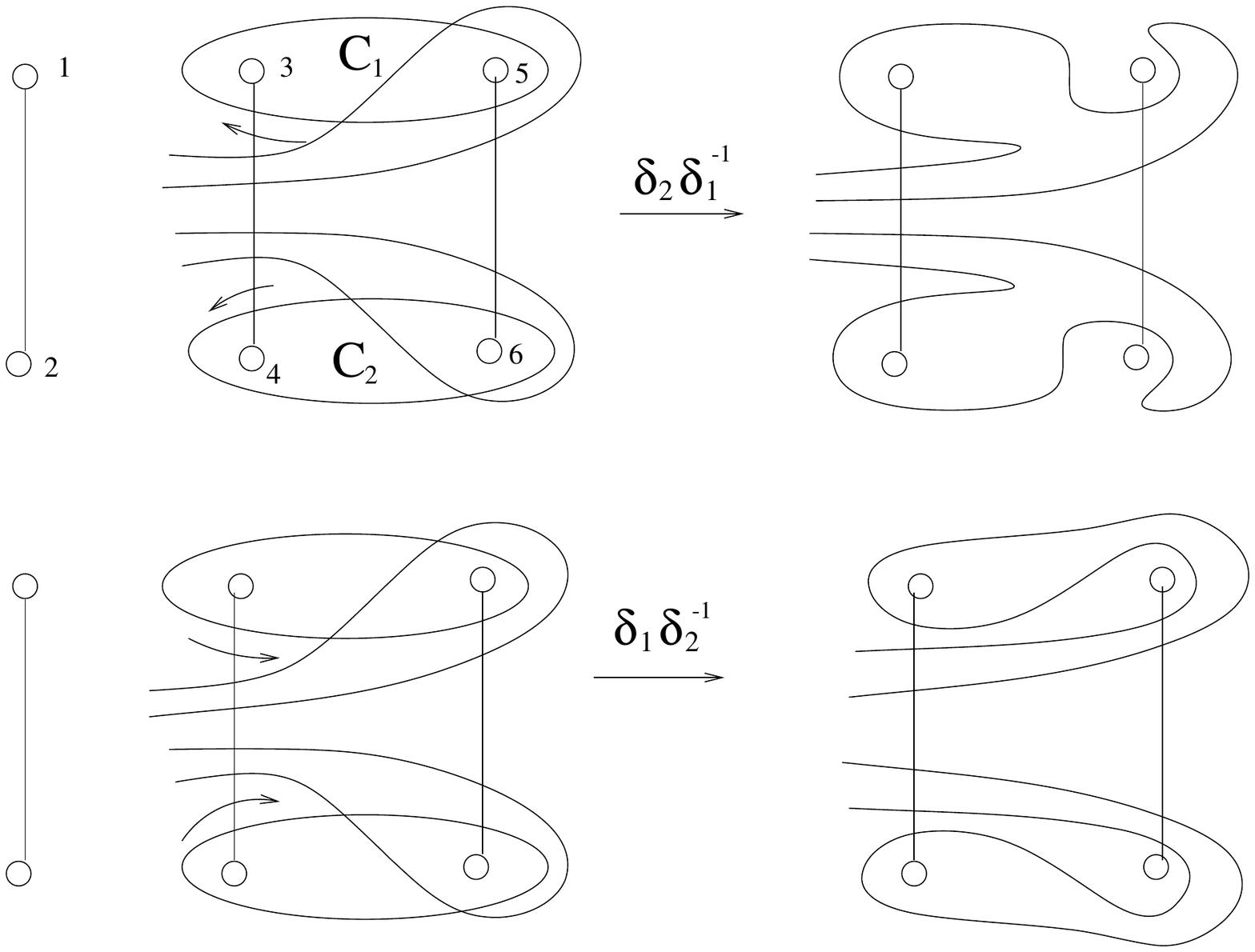}
\vskip -200pt
\caption{}\label{p8}
\end{figure}

\begin{Lem}[\cite{3}]\label{T34}
$\gamma_0$ bounds an essential disk in $B^3-\epsilon$ if and only if $\delta_3^{\pm 1}(\gamma_0)$  does, where $\delta_3$ be a half Dehn twist counter clockwise supported on the punctured disk $E_4'$. (Refer to Figure~\ref{p9}.)
\begin{figure}[htb]
\includegraphics[scale=.6]{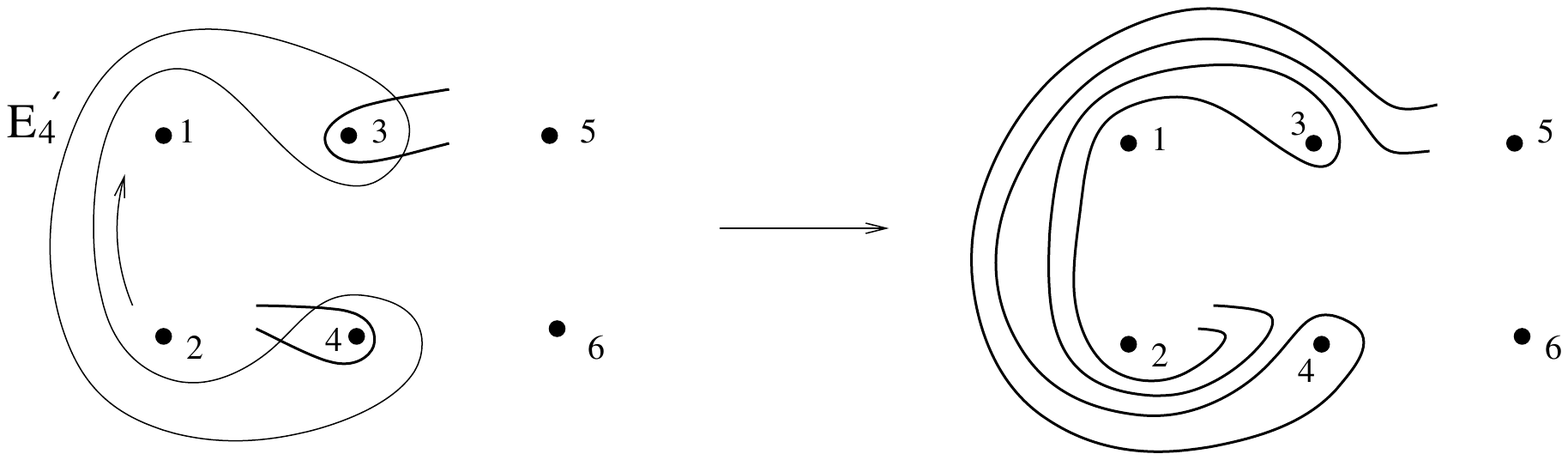}
\vskip -400pt
\caption{}\label{p9}
\end{figure}
\end{Lem}

\begin{Lem}\label{T35}
Suppose that $\gamma_0$ is in standard position with $m_3>0$. If $m_{2'}=m_{8'}=m_{11'}=0$ and $m_{8''}+m_{11''}>0$, then $\gamma_0$ does not bound an essential disk in $B^3-\epsilon$. 
\end{Lem}

\begin{proof}
Suppose that $\gamma_0$ is in standard position with $m_3>0 $ and $m_{2'}=m_{8'}=m_{11'}=0$ and $m_{8''}+m_{11''}>0$. \\ 

Now, we consider several subcases for this. We note that $m_9>0$. If not, then there is no output either from $E_1''^+$ or from $E_3''^-$. Then, it contradicts the connectivity of arcs in $E_1''$ or $E_3''$.\\

 First, assume that $m_{2''}=0$.\\

We claim that $m_{10''}>0$ and $m_{8''}>0$ if $m_{11''}=0$. The proof is as follows; If both $m_{10''}$ and $m_{8''}$ are zero, then $\gamma_0$ cannot be a simple closed curve because of the connectivity of arcs in $E_3''$.
If $m_{10''}>0$ and $m_{8''}=0$ ($m_{11''}=0$) then the uppermost arc of type $10''$ and the lowermost arc of type 9 make the simple closed curve $\gamma_0$. We easily can check that it does not bound an essential disk in $B^3-\epsilon$ by the fundamental group argument. If $m_{8''}>0$ and $m_{10''}=0$ ($m_{11''}=0$) then  we can take the uppermost arc of type 9 and the lowermost arc of type $8''$ and the outermost arc of type 3 which  make the simple closed curve $\gamma_0$.
 However, we can check that $\gamma_0$ does not bound an essential disk in $B^3-\epsilon$ because $\gamma_0$ encloses exactly three punctures in $\Sigma_{0,6}$ and it is impossible to bound an essential disk in $B^3-\epsilon$. Therefore,  $m_{10''}>0$ and $m_{8''}>0$ if $m_{11''}=0$.\\
 
 \begin{figure}[htb]
\includegraphics[scale=.7]{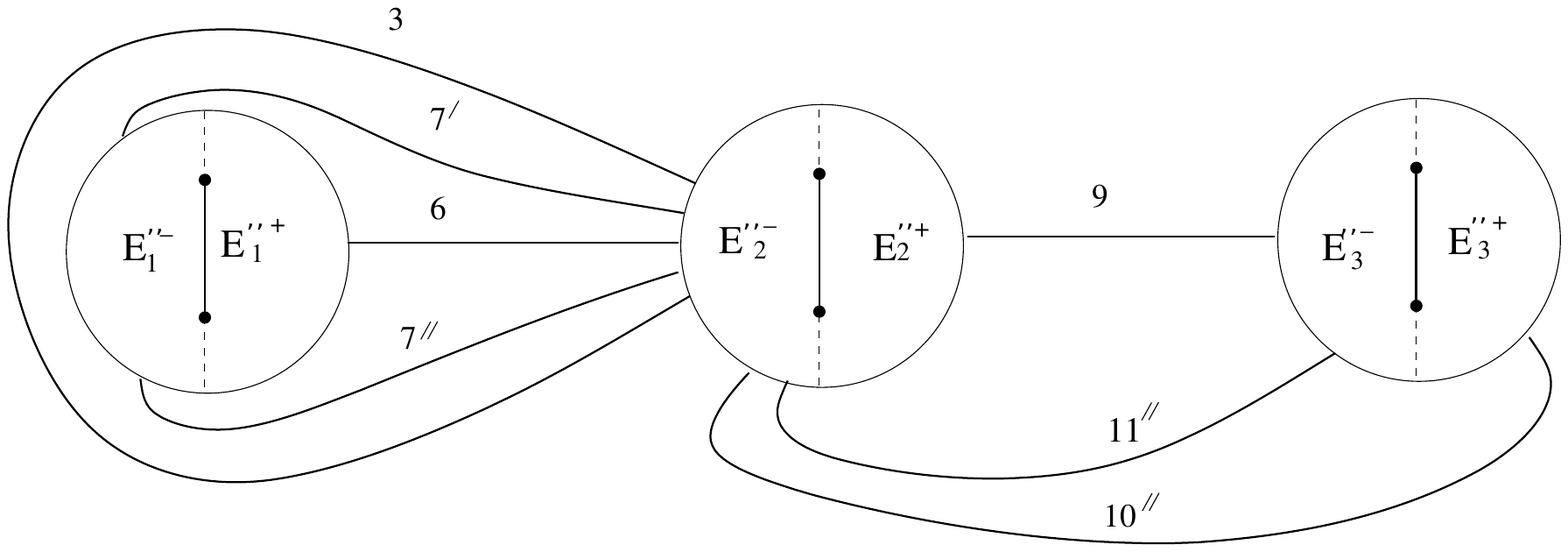}
\vskip -450pt
\caption{}
\label{p10}
\end{figure}

 We also note that $m_{10''}>0$ and $m_{11''}>0$ if $m_{8''}=0$ by a similar argument in the previous claim. It is impossible that  both $m_{8''}$ and $m_{11''}$ are positive since they intersect. So, we have the following two cases.\\

Case 1. Suppose that $m_{8''}=0$. Then we have the diagram as in Figure~\ref{p10}. So, we have two conditions $m_{10''}=m_{11''}+m_9$ for $E_3''$ and $m_9>m_{10''}+m_{11''}$ for $E_1''$. This implies that $m_9>m_{11''}+m_9+m_{11''}\geq m_9$. It makes a contradiction.\\

Case 2. Suppose that $m_{11''}=0$. Then  we have the diagram as in Figure~\ref{p11}.  
Let $m_{8''}=a$ and $m_{10''}=b$. Then we can consider the three subcases as follows.

\begin{figure}[htb]
\includegraphics[scale=.7]{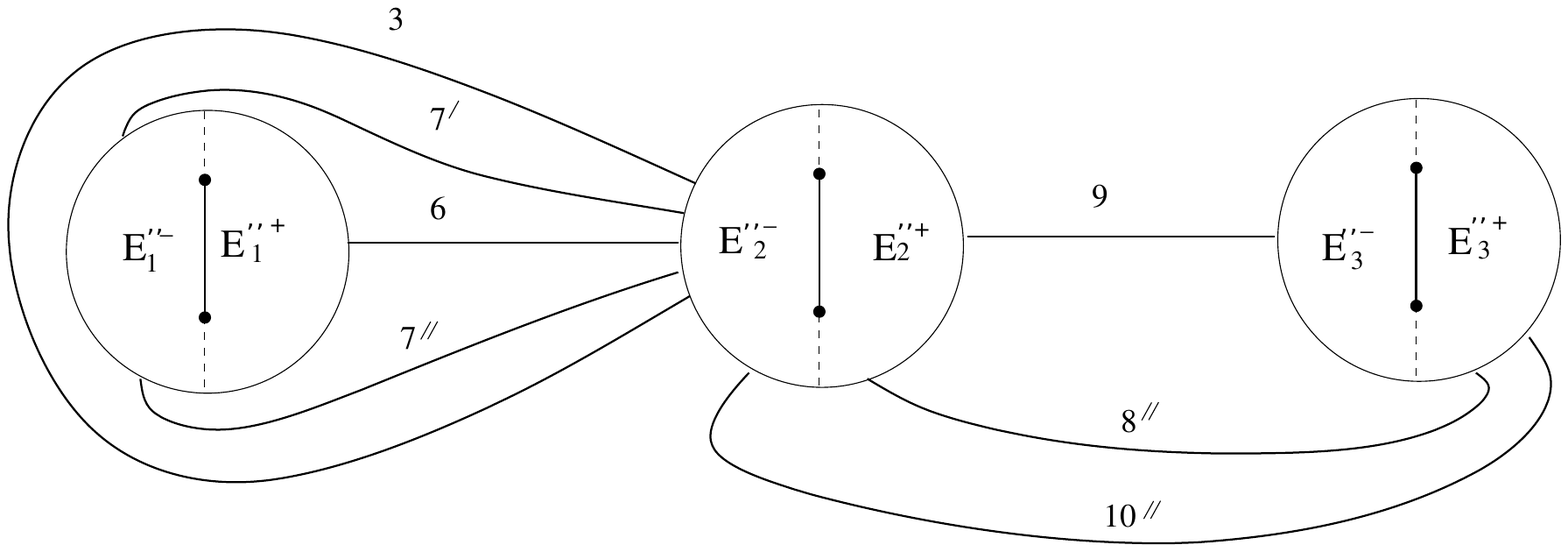}
\vskip -450pt
\caption{}
\label{p11}
\end{figure}
\begin{enumerate}
\item $m_{8''}= m_{10''}$ : Apply the homeomorphism $\delta_1^{-1}\delta_2$ to $\gamma_0$. Then we have a case that has $m_2=m_{10}=m_{11}=0$. However, we already knew that $\delta_1^{-1}\delta_2(\gamma_0)$ does not bound an essential disk in $B^3-\epsilon$. By Lemma~\ref{T33}, we see that $\gamma_0$ also does not bound an essential disk in $B^3-\epsilon$.\\

\item $m_{8''}<m_{10''}$ : We also apply the homeomorphism $\delta_1^{-1}\delta_2$ to $\gamma_0$. Then we have the same diagram as in Figure~\ref{p11}. Let $m'_{i}$ be the weights of the newly obtained diagram by  applying $\delta_1^{-1}\delta_2$. Then we have $m'_{8''}=m_{8''}$ and $m'_{10''}=m_{10''}-m_{8''}$. If $m'_{10''}\geq m'_{8''}$, then we return to one of the cases $(1)$ and $(2)$. If $m'_{10''}< m'_{8''}$, then we move to  the next case.\\

\item $m_{8''}>m_{10''}$ : We apply the homeomorphism $\delta_1^{-1}\delta_2$ to $\gamma_0$. Then 
$\delta_1^{-1}\delta_2(\gamma_0)$ has $m'_{8''}=m_{10''}, m'_{10''}=0$ and $m'_{2''}=m_{8''}-m_{10''}>0$. So, we move to the case that $m_{2''}>0$.\\
\end{enumerate}

Second, we assume that $m_{2''}>0$.\\

 Then, we notice that $m_{10''}=m_{11''}=0$ since the arcs for $m_{10''}$ and $m_{11''}$ intersect with the arc for $m_{2''}$. Let $m_{2''}=a, m_3=b$ and $m_{8''}=c$ for convenience. Also, let $m'_i$ be the weights of the new diagram by applying $\delta_3$ or $\delta_1^{-1}\delta_2$.
\begin{figure}[htb]
\includegraphics[scale=.9]{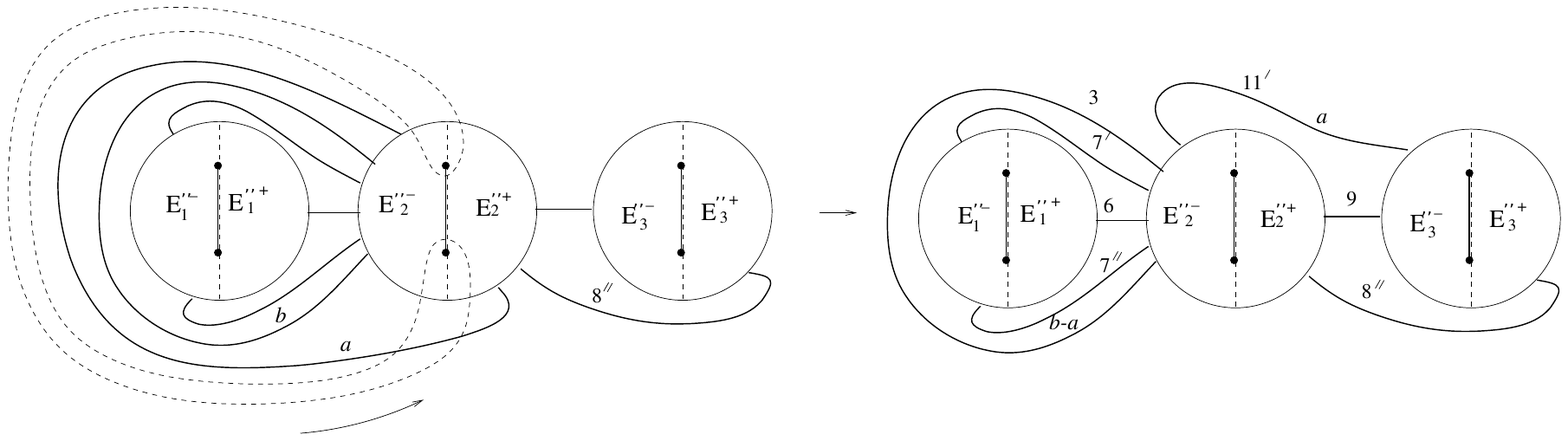}
\vskip -630pt
\caption{$m_{2''}>0$}
\label{p12}
\end{figure}

We apply the half Dehn twist $\delta_3$ to $\gamma_0$. Here, we note that $b>a$. If not, then $\delta_3(\gamma_0)$ has $m'_1+m'_3=0$. It violates Lemma~\ref{T1} and Lemma~\ref{T32}.\\

Then $\delta_3(\gamma_0)$ has the new weights $m'_3=b-a$ and $m'_{11'}=a$ as the second diagram of Figure~\ref{p12}.
Now, we consider the three subcases.

\begin{enumerate}
\item $a = c$ : Then $m'_9 = 0$. This implies that $m'_3 = 0$ and it contradicts the condition  to bound an essential disk in $B^3-\epsilon$.\\

\item $a > c$: It violates the condition $c\geq a$ for the connectivity in $E_3''$.\\

\begin{figure}[htb]
\includegraphics[scale=.9]{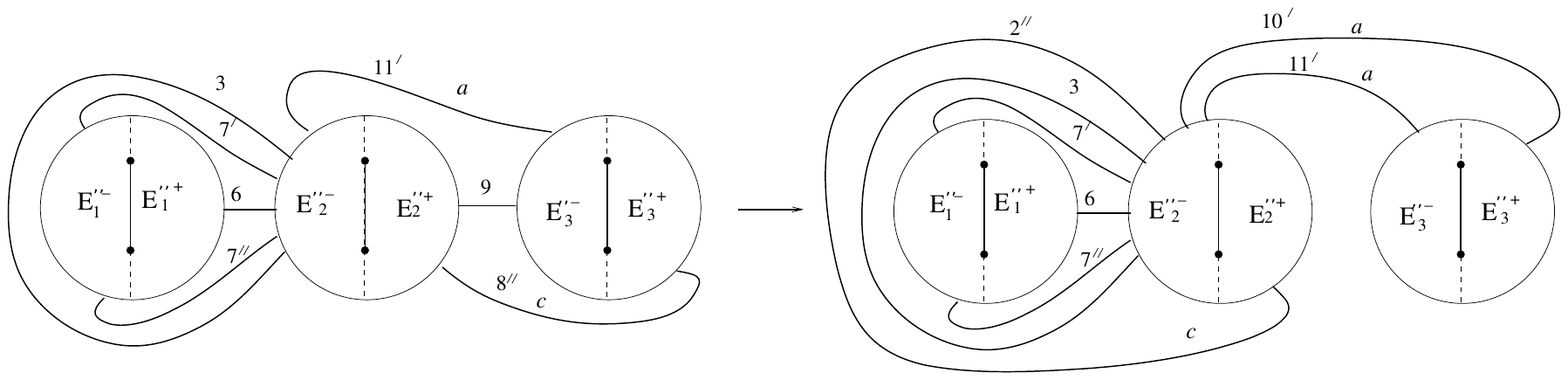}
\vskip -630pt
\caption{}
\label{p13}
\end{figure}

\item $a<c$: Apply the homeomorphism $\delta_1^{-1}\delta_2$ to $\delta_3(\gamma_0)$. Let $\gamma_1=\delta_1^{-1}\delta_2(\delta_3(\gamma_0))$. Let $m''_i$ be the weights for $\gamma_1$. Then we have $m''_3=b-a > a$. Otherwise, $\delta_3(\gamma_1)$ does not bound an essential disk in $B^3-\epsilon$. This implies that $\gamma_1$ also does not bound an essential disk in $B^3-\epsilon$.  We also have $m_{2''}=c$ and $m_{10'}=m_{11'}=a$ for $\gamma_1$ as the second diagram of Figure~\ref{p13}. This implies that $m_9=0$ because of the connectivity in $E_3''$. Therefore, $m_3=0$ because of the connectivity in $E_2''$. This violates the condition that $m_3>0$. Therefore, if $\gamma$ has $m_{2''}>0$ then $\gamma$ does not bound an essential disk in $B^3-\epsilon$.\\

 These complete this lemma. 

\end{enumerate}
\end{proof}

Now, we will prove the main theorem.

\begin{Thm}\label{T}
Suppose that two rational 3-tangle diagrams $T_F, T_G\in \mathcal{T}_{\{\hat{\sigma_1},\hat{\sigma_2},\hat{\sigma_3}\}}$.  $T_F\approx T_G$ if and only if   $[G^{-1}F(\partial E_1)]= [\partial E_1]$ and $(G^{-1}F(\epsilon_2)\cup G^{-1}F(\epsilon_3),B^3)\approx (\epsilon_2\cup \epsilon_3,B^3)$.
\end{Thm}

\begin{proof}
\begin{figure}[htb]
\includegraphics[scale=.9]{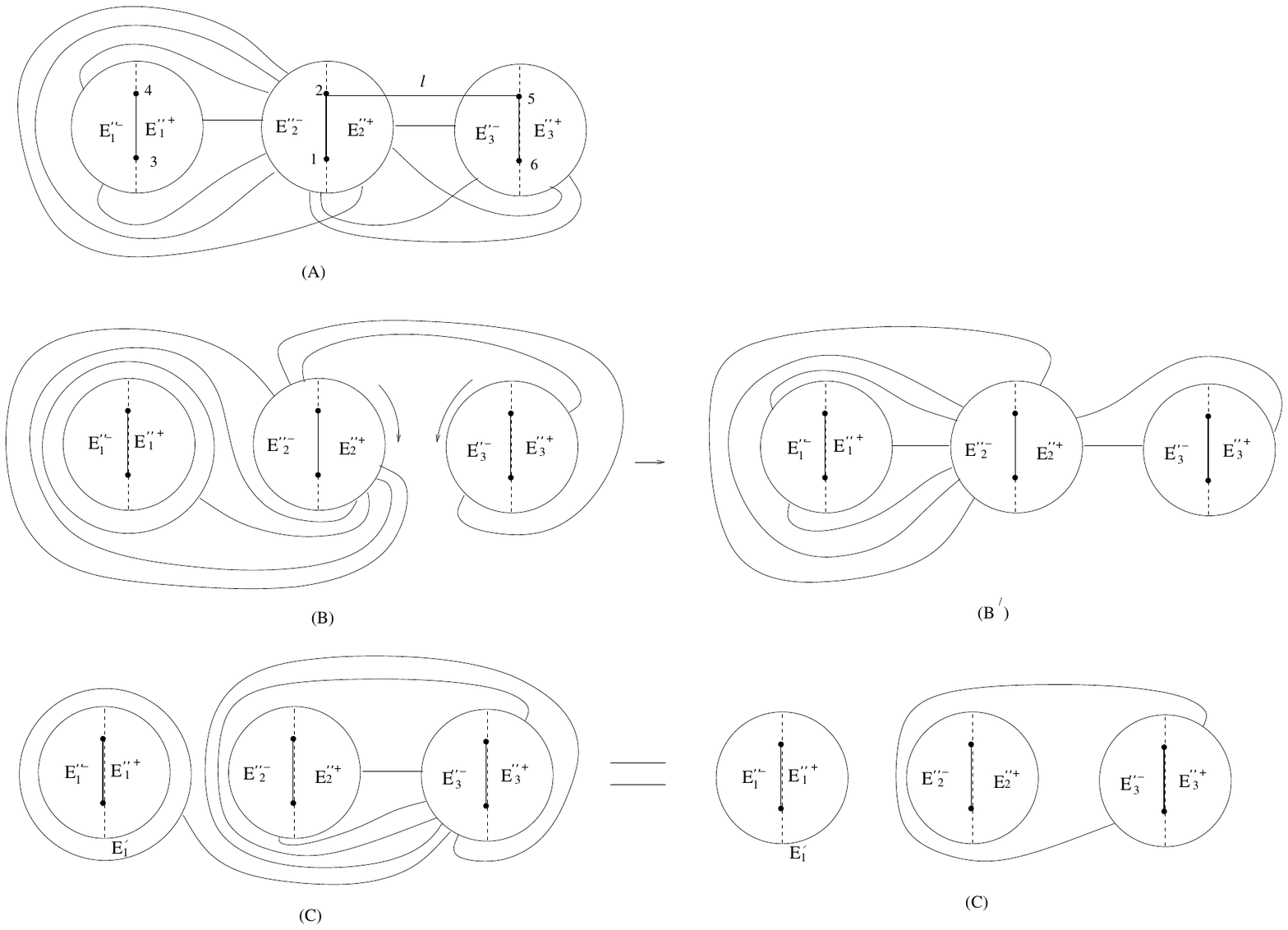}
\vskip -410pt
\caption{}
\label{p14}
\end{figure}

Since $T_F\approx T_G$, we know that $T_{G^{-1}F}\approx T_{id}$. By Lemma~\ref{T2}, $G^{-1}F(\partial E_1)$ bounds an essential disk in $B^3-\epsilon$. So, $G^{-1}F(\partial E_1)$ is isotopic to either $\partial E_2$ or a simple closed curve $\gamma_0$ which has $x_{ii}>0$ for some $i$. Let $l$ be a simple arc between the puncture $2$ and the puncture $5$ as in the first diagram of Figure~\ref{p14}.
Let $|A\cap B|$ be the minimal intersection number between $A$ and $B$.\\

Claim: $|\gamma_0\cap l|=0$.\\

\textit{Proof of Claim.} 
 We note that $|\partial E_1\cap l|=0$. For a simple closed curve $\gamma$, if $|\gamma\cap l|=0$ then $|\sigma_i^{\pm 1}(\gamma)\cap l|=0$ for $i=1,2,3$. This completes the proof of this claim. \\

Now, we has some possible cases $(A),(B)$ and $(C)$ for $\gamma_0$ which has $|\gamma_0\cap l|=0$.\\

We note that any simple closed curve for the cases $(A)$ and $(B')$ cannot bound an essential disk in $B^3-\epsilon$ by using the similar argument in Lemma~\ref{T35}. In the case $(B)$, we get $(B')$ from $(B)$ by applying half Dehn twists supported on $E_1''$ and $E_3''$ with the given direction as in the diagram $(B)$.\\

Also, it is not difficult to show that the any simple closed curves for the case $(C)$ cannot bound an essential disk in $B^3-\epsilon$ by considering the connectivity of arcs in $E_i'$. Therefore, $G^{-1}F(\partial E_1)$ should be isotopic to $\partial E_1$.\\

Now, it is clear that if  $T_F\approx T_G$ then   $[G^{-1}F(\partial E_1)]= [\partial E_1]$ and $(G^{-1}F(\epsilon_2)\cup G^{-1}F(\epsilon_3),B^3)\approx (\epsilon_2\cup \epsilon_3,B^3)$ by the previous argument.\\

To show the other direction, assume that $T_F\not\approx T_G$ and $[G^{-1}F(\partial E_1)]= [\partial E_1]$. We will show that $(G^{-1}F(\epsilon_2)\cup G^{-1}F(\epsilon_3),B^3)\not\approx (\epsilon_2\cup \epsilon_3,B^3)$\\

We consider $G^{-1}F(\partial E_3)$. We note that $G^{-1}F(\partial E_3)$ is non-parallel, disjoint with $G^{-1}F(\partial E_1)$ which is isotopic to $\partial E_1$ since $\partial E_1$ and $\partial E_3$ are disjoint. \\

Consider the band sum as in Fugure~\ref{p15}. We note that the band sum slides over the two punctures $\{1,2\}$.\\

\begin{figure}[htb]
\includegraphics[scale=.7]{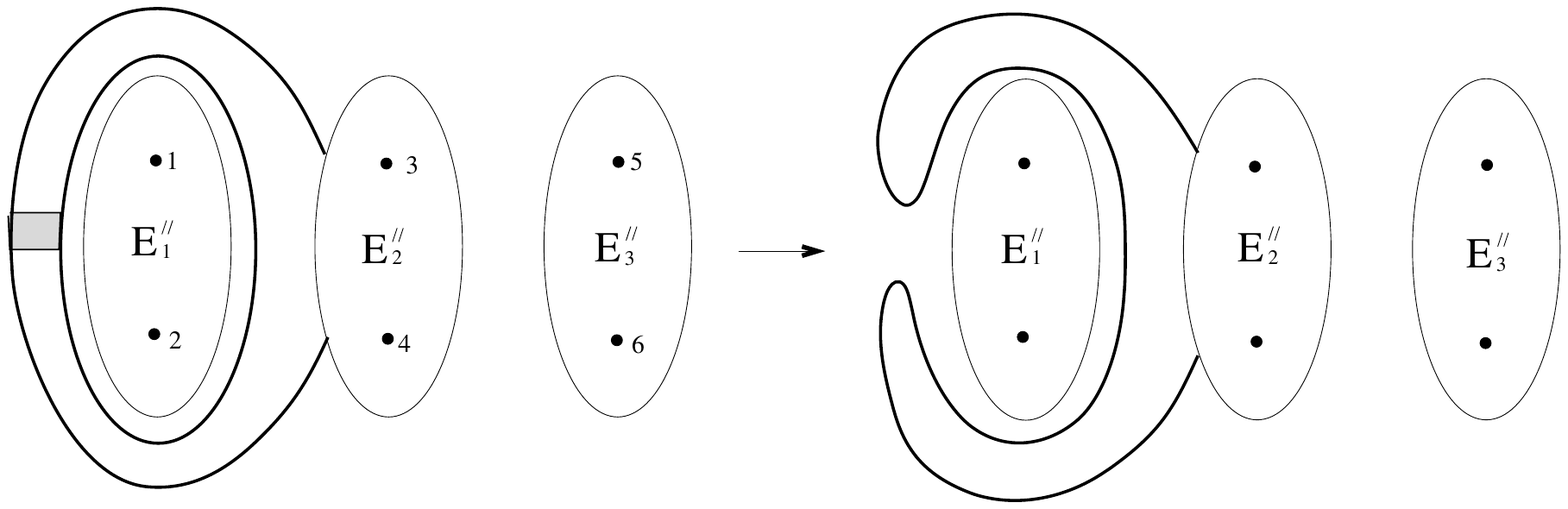}
\vskip -450pt
\caption{}
\label{p15}
\end{figure}

Also, by Lemma~\ref{T4}, $G^{-1}F(\partial E_3)+_R\partial E_1$ is non-parallel, disjoint with both $G^{-1}F(\partial E_3)$ and $\partial E_1$.
Especially, since $G^{-1}F(\partial E_3)$ is disjoint, non-parallel with $\partial E_1$ up to isotopy, we can use the band sum anytime we want. This implies that we can ignore the two punctures $\{1,2\}$ when we check whether or not $G^{-1}F(\partial E_3)$ bounds an essential disk in $B^3-\epsilon$. Also, we note that there is unique simple closed curve which bounds an essential disk in $B^3-(\epsilon_2\cup \epsilon_3)$ and especially it is isotopic to $\partial E_3$.\\

 So, if $(G^{-1}F(\epsilon_2)\cup G^{-1}F(\epsilon_3),B^3)\approx (\epsilon_2\cup \epsilon_3,B^3)$ then we note that  $G^{-1}F(\partial E_3)$ does bound an essential disk in $B^3-\epsilon$ since $\partial E_1$ does. Then we have $T_{G^{-1}F}$ is the $\infty$-tangle diagram.
 This implies that $T_F\approx T_G$. This contradicts the assumption that $T_F\not\approx T_G$.
 Therefore, $(G^{-1}F(\epsilon_2)\cup G^{-1}F(\epsilon_3),B^3)\not\approx (\epsilon_2\cup \epsilon_3,B^3)$ and it completes the proof.\\

 This completes this Theorem.

\end{proof}

 \begin{Cor}\label{C}
Suppose that two rational 3-tangle diagrams $T_F, T_G\in \mathcal{T}_{\{\hat{\sigma_1},\hat{\sigma_2},\hat{\sigma_3}\}}$.  $T_F\sim T_G$ if and only if  $[F(\partial E_1)]= [G(\partial E_1)]$ and $(F(\epsilon_2)\cup F(\epsilon_3),B^3)\approx (G(\epsilon_2)\cup G(\epsilon_3),B^3)$.
\end{Cor}

\begin{proof}
It is a modification of Theorem~\ref{T}.
\end{proof}

\end{document}